\documentclass[12pt]{amsart}

\usepackage[top=4cm,bottom=4cm,inner=4cm,outer=4cm]{geometry}
\usepackage{comment}
\newcommand{\disp}{\displaystyle}
\newcommand{\nc}{\newcommand}
\nc{\G}{{\Gamma}} \nc{\BC}{{\mathbb C}} \nc{\BQ}{{\mathbb Q}}
\nc{\BR}{{\mathbb R}} \nc{\BZ}{{\mathbb Z}} \nc{\BP}{{\mathbb P}} \nc{\PC}{{\BP_1(\BC)}}
\nc{\BN}{{\mathbb N}} \nc{\BM}{{\mathbb M}}
\nc{\fH}{{\mathbb H}}

\nc{\mat}{{\binom{a\,\ b}{c\,\ d}}}
\nc{\U}{{\mathcal U}}
\nc{\PS}{{\mbox{PSL}_2(\BZ)}} \nc{\SL}{{\mbox{SL}_2(\BZ)}}
\nc{\SR}{{\mbox{SL}_2(\BR)}} \nc{\PR}{{\mbox{PSL}_2(\BR)}}
\nc{\SLC}{{\mbox{SL}_2(\BC)}}

\nc{\GL}{{\mbox{GL}}} \nc{\PQ}{{\mbox{PGL}_2^+(\BQ)}}
\nc{\GR}{{\mbox{GL}_2^+(\BR)}} \nc{\PG}{{\mbox{PGL}_2(\BC)}}
\nc{\GC}{{\mbox{GL}_2(\BC)}}
\nc{\f}{{\mathcal{F}(\fH)}}
\nc{\Cc}{\widehat{\BC}}
\nc{\e}{{E_{\rho}(\G)}}
\nc{\g}{{\gamma}}
\nc{\vm}{{V_{\rho}(\G)}}
\nc{\oo}{{\mathcal O}}
\nc{\M}{{\mbox{M}}}
\nc{\om}{{\omega}}
\nc{\Om}{{\Omega}}
\nc{\TX}{{\widetilde{X}}}
\nc{\ol}{\overline}
\nc{\cl}{{\mathcal L}}
\nc{\ce}{{\mathcal E}}
\nc{\la}{{\lambda}}
\nc{\La}{{\Lambda}}

\nc{\cz}{{\mathcal Z}}

\newtheorem{numbered}{}[section]
\newtheorem{thm}[numbered]{Theorem}

\newtheorem{lem}[numbered]{Lemma}
\newtheorem{remark}[numbered]{Remark}
\newtheorem{prop}[numbered]{Proposition}
\newtheorem{cor}[numbered]{Corollary}

\numberwithin{equation}{section}

\newcommand{\propref}[1]{Proposition~\ref{#1}}
\newcommand{\secref}[1]{\S\ref{#1}}
\newcommand{\lemref}[1]{Lemma~\ref{#1}}
\newcommand{\corref}[1]{Corollary~\ref{#1}}

\begin{document}

\title[]{Equivariant solutions to modular Schwarzian equations}
\author[]{Hicham Saber} \author[]{Abdellah Sebbar}
\address{Department of Mathematics, Faculty of Science, University of Ha'il,   Ha'il, Kingdom of Saudi Arabia}
\address{Department of Mathematics and Statistics, University of Ottawa,
	 Ottawa Ontario K1N 6N5 Canada}
\email{hicham.saber7@gmail.com}
\email{asebbar@uottawa.ca}
\maketitle
\maketitle
\begin{abstract} For every positive integer $r$, we solve the modular Schwarzian differential equation $\{h,\tau\}=2\pi^2r^2E_4$, where $E_4$ is the weight 4 Eisenstein series,  by means of equivariant functions on the upper half-plane. This paper supplements previous works \cite{forum, ramanujan}, where the same equation has been solved for infinite families of rational values of $r$. This also leads  to the solutions to the modular differential equation $y''+r^2\pi^2E_4\,y=0$ for every positive integer $r$. These solutions are quasi-modular forms for $\SL$ if $r$ is even or for the subgroup of index 2, $\SL^2$, if $r$ is odd.
\end{abstract}

\section{Introduction}

For a meromorphic function $f$ on a domain $D\subset\BC$, the Schwarz derivative is defined by
\[
\{f,z\}=\left(\frac{f''}{f'}\right)'-\frac{1}{2}\left(\frac{f''}{f'}\right)^2.
\]
It is protectively invariant in the sense that for two meromorphic functions
 $f$ and $g$ on $D$,
\[
\{f,z\}\,=\,\{g,z\} \mbox{ if and only if } g\,=\,\frac{a\,f+b}{c\,f+d} \mbox{ for some } \mat\in\GC.
\]
In fact, $\{f,z\}$ is the differential operator with the least order with this invariance property. Moreover, $\{f,z\}=0$ if and only if $f$ is a linear fractional transformation, that is an element of $\PG$. In addition, the Schwarz derivative behaves as a quadratic differential satisfying the cocyle condition
\[
\{f,z\}dz^2=\{f,w\}dw^2+\{w,z\}dz^2.
\]
Most of the functional properties satisfied by the Schwarz derivative  are derived from the following close connection with second order ordinary differential equations:\\
Let $R(z)$ be a meromorphic function on $D$, and let $y_1$ and $y_2$ be two linearly independent  local solutions to
\[
y''+R(z)y=0,
\]
then $h=y_2/y_1$ satisfies
\[
\{h,z\}=2R(z).
\]
The systematic study of the automorphic properties of the Schwarz derivative began in \cite{mathann}. If $D=\fH$, the upper half of the complex plane, and if $f$ is an automorphic function for a discrete subgroup $\G$ of $\SR$, then $\{f,\tau\}$, $\tau\in\fH$, is a weight 4 automorphic form for the normalizer of $\G$ in $\SR$. As an example, if $\lambda$ is the Klein modular elliptic function for the principal congruence group $\Gamma(2)$, and $E_4$ is the weight 4 Eisenstein series, then
\[
\{\lambda,\tau\}\,=\frac{\pi^2}{2}\,E_4(\tau).
\]
Conversely, if $f$ is a meromorphic function on $\fH$ such that $\{f,\tau\}$ is a weight 4 automorphic function for a discrete subgroup $\G$ of $\SR$, one can show that there exists a 2-dimensional complex representation $\rho$ of $\Gamma$ such that
\[
f(\gamma\cdot\tau)\,=\,\rho(\gamma)\cdot f(\tau)\,,\ \tau\in\fH\,,\ \gamma\in\G,
\]
where the action on both sides is by linear fractional transformation \cite{forum}. We say that $f$ is $\rho-$equivariant. If $\rho=1$ is constant then $f$ is an automorphic function (of weight 0), while if $\rho=\mbox{Id}$ is the embedding of $\G$ in $\SR$, we simply say that $f$ is an equivariant function, that is, it commutes with the action of $\G$.

To make the problem more precise, we suppose that $\G=\SL$ and that $\{f,\tau\}$ is a holomorphic weight 4 modular form. In other words, we are dealing  with the following equation
\[
\{f,\tau\}\,=\,s\,E_4(\tau),
\]
where $s$ is a complex parameter. With the help of the Frobenius method, it was established in \cite{forum} that for a solution $f$ to be either meromorphic or to have a logarithmic singularity at $\infty$, we must have $s=2\pi^2r^2$ where $r\in\BQ$. The paper {\em loc.\,cit.} determines exactly for which rational number $r$, the solution $f$ is a modular function for a finite index subgroup of $\SL$. In other words, looking at $f$ as a $\rho-$equivariant function, we seek $r\in\BQ$ such that $\ker \rho$ has a finite index in $\SL$. A necessary condition is that the representation $\rho$ has to be irreducible and thus its image, which is finite, can be classified explicitly. In the meantime, we have two coverings of compact Riemann surfaces ; one is given by
\[
f:X(\ker\rho)\longrightarrow X(\SL)
\] and the other is induced by the natural inclusion $\ker\rho\subseteq\SL$: \[
\pi:X(\ker\rho)\longrightarrow X(\SL)
.\]
 Here $X(\G)$ is the modular curve attached to the subgroup $\G$. Let $m$ and $n$ be the respective ramification indices above $\infty$ for the two coverings. Then  the Riemann-Hurwitz formulas for the two covering imply that $\ker\rho$ is a normal, genus zero and torsion-free subgroup of $\SL$, $s=\pi^2(m/n)^2$ and $\ker \rho=\G(n)$ for $2\leq n\leq 5$. 

In the meantime, the cases for which the solutions  to $\{f,\tau\}=sE_4$ correspond to reducible representations were determined in \cite{ramanujan}. This occurs precisely when $s=2\pi^2r^2$ with $r=m/6$ and $m\equiv 1\mod 12$. Moreover, if $m=12k+1$, then a solution is given by
\[
f=\int_i^{\tau}\,f_k(z)\,dz,
\]
where $f_k$ is the weight 2 modular form given in terms of the Dedekind eta-function and the $J-$function by
\[
f_k=\frac{\eta^4}{\prod_{i=1}^{k}(J-a_i)^2}.
\]
The coefficients $a_i$ satsify the algebraic system
\[
\frac{3}{1-a_i}-\frac{4}{a_i}=\sum_{j\neq i}\frac{12}{a_i-a_j}\,,\ \ 1\leq j\leq k,
\]
which has a unique solution $(a_1,\ldots,a_n)\in(0,1)^n$ up to a permutation.

In all the above cases, a solution $f$ to the Schwarzian equation $\{f,\tau\}=sE_4$ is necessarily locally univalent all over $\fH$ and thus leads to two linearly independent solutions to the modular differential equation
\[
y''+\frac{s}{2}E_4\,y=0
\]
given by $y_1=1/\sqrt{f'}$ and $y_2=f/\sqrt{f'}$. In particular, in the reducible case, a solution to
\[
y''+\pi^2\left(\frac{12k+1}{6}\right)^2\,E_4\,y\,=\,0
\]
is given by
\[
y\,=\,\eta^{-2}\,\prod_{i=1}^{k}(J-a_i),
\]
where the $a_i$'s, $1\leq i\leq k$, satisfy the above algebraic system. The differential equation $y''+(\pi^2/36) \,E_4\,y=0$, corresponding to the particular case $k=0$, was considered by Klein in \cite{klein} and Hurwitz wrote down a solution as $y=\eta^{-2}$ \cite{hurwitz}.

In this paper we shall describe how to solve the equation $\{f,\tau\}=2\pi^2r^2 E_4$ for every positive integer $r$. In particular we will show that it always admits an equivariant solution. In order to construct these solutions we focus on the ordinary differential equation
$y''+\pi^2r^2\,E_4\,y=0$ for $r\in\BN$ which turns out to admit solutions that are quasi-modular forms. To explicitly determine these solutions, we proceed as follows: Let $\G=\SL$ if $r$ is even and if $r$ is odd we let $\G=\SL^2$ which is the unique subgroup of index 2 in $\SL$ whose elements are the squares of elements of the modular group. In each case, $\G$ has genus 0, has one cusp, and is generated by two elliptic elements.
 We start with a weight $-2$ modular form $g$  that is holomorphic on $\fH$ and has a pole at infinity. Then the function defined by
 \[
 f_0(\tau)=\int_i^{\tau}\,g(z)E_4(z)\,dz
 \]
is necessarily a modular function for $\G$. Moreover, $f_1=\pi^2r^2f_0+g'$ is a quasi-modular form of weight 0 and depth 1 such that $F(\tau)=f_1''+\pi^2r^2f_1$ is
a weight 4 modular form that is holomorphic on $\fH$. Furthermore, one can prove that for each $r$, there exists  $g$ such that $F$ vanishes at $\infty$ and thus must be zero everywhere. In other words, $f$ is a solution to the  equation $y''+r^2\pi^2E_4y=0$. The construction of $g$ is carried out by solving a simple eigenvector problem for a triangular matrix of size $r$ if it is odd, and of size $r/2$ if $r$ is even. A second solution is given by $f_2(\tau)=-2g(\tau)+\tau f_1(\tau)$. Furthermore, a solution to the Schwarz equation $\{f,\tau\}=2\pi^2r^2E_4$ is given by
\[
f(\tau)\,=\,\tau-\frac{2g(\tau)}{f_1(\tau)}
\]
which is equivariant for $\G$. Finally, we provide examples to illustrate this method for small values of $r$ and we determine the explicit solutions to both differential equations.

It is worth mentioning that in recent years, the subject of modular differential equations has gained a lot of  interest both in number theory \cite{grabner,ka-ko,ka-za,ka-et-al,mason,mason-franc, nakaya} and in mathematical physics \cite{milas,mukhi} to mention a few.

\section{Some notations and definitions}\label{notation}
We recall some definitions and facts about classical modular forms that will be needed in this paper. The maine references are \cite{ram,rankin,shimura}.

Define the upper half-plane by \(\fH=\{\tau\in\BC\,\mid\,\mbox{Im}\,{\tau>0}\}.\)
The Eisenstein series $E_2$, $E_4$ and $E_6$ are defined by their $q-$expansions with $q=\exp(2\pi i\tau)$:
\begin{align*}
E_2(\tau)&=1-24\,\sum_{n\geq 1}\,\sigma_1(n)\,q^n\,,\\
E_4(\tau)&=1+240\,\sum_{n\geq 1}\,\sigma_3(n)\,q^n\,,\\
E_6(\tau)&=1-504\,\sum_{n\geq 1}\,\sigma_5(n)\,q^n\,.
\end{align*}
Here, the arithmetic function $\sigma_k(n)$ is defined as the sum of the $k-$th powers of the positive divisors of $n$. The functions $E_4$ and $E_6$ are modular forms of respective weights 4 and 6, while $E_2$ is a quasi-modular form of weight 2 and depth 1.

The Dedekind eta-function is defined by
\[
\eta(\tau)\,=\,q^{\frac{1}{24}}\,\prod_{k\geq1}(1-q^n)\,,
\]
and the discriminant $\Delta$, which is a cusp form of weight 12, is defined by
\[
\Delta(\tau)\,=\,\eta(\tau)^{24}\,=\,\frac{1}{1728}(E_4(\tau)^3-E_6(\tau)^2).
\]
The elliptic modular function $J$  is defined by
\[
J(\tau)\,=\,\frac{1}{1728}\frac{E_4(\tau)^3}{\Delta}.
\]
The Klein elliptic modular function $\lambda$ for $\G(2)$ is defined by
\[
\lambda(\tau)\,=\,\left(\frac{\eta(\tau/2)}{\eta(2\tau)}\right)^8.
\]
We also define the Jacobi theta-functions
\begin{align*}
	\theta_2(\tau)&=\sum_{n\in\BZ}\,q^{\frac{1}{2}(n+\frac12)^2},\\
	\theta_3(\tau)&=\sum_{n\in\BZ}\,q^{\frac{1}{2}n^2},\\
	\theta_4(\tau)&=\sum_{n\in\BZ}\,(-1)^nq^{\frac{1}{2}n^2}\,
\end{align*}
satisfying the Jacobi identity
\[
\theta_2^4+\theta_4^4=\theta_3^4.
\]
Finally, we have the Ramanujan identities:
\begin{align*}
	\frac{1}{2\pi i}\,\Delta'&=E_2\Delta\\
	\frac{1}{2\pi i}\,E_2'&=\frac{1}{12}(E_2^4-E_4)\\
	\frac{1}{2\pi i}\,E_4'&=\frac{1}{3}(E_2E_4-E_6)\\
	\frac{1}{2\pi i}\,E_6'&=\frac{1}{2}(E_2E_6-E_4^2).
\end{align*}

\section{Automorphic Schwarzian equations and equivariant functions}

For any pair $(\G,\rho)$ where $\G$ is a discrete subgroup of $\SR$ and $\rho$ a 2-dimensional complex representation of $\ol{\G}$, the image of $\G$ in $\PR$, the existence of $\rho-$equivariant functions was established in \cite{kyushu}.
Furthermore, if $\disp F=\binom{f_1}{f_2}$ is a vector-valued automorphic form for $\G$ of weight $k$ and multiplier system $\rho$, then $h_F=f_2/f_1$ is a $\rho-$equivariant function for $\G$ and the map $F\mapsto h_F$ is a surjective map from the set of $2-$dimensional vector-valued automorphic forms for $\G$ of arbitrary weight $k$ and multiplier system $\rho$ and the set of $\rho-$equivariant functions for $\G$ \cite{vvmf}.

 The case of equivariant functions, that is when the representation $\rho=\mbox{Id}$ is the defining representation for $\G$ is very interesting in its own. They will be the main focus of this paper as they will characterize the solutions to the Schwarzian equations under consideration. A trivial example is $h(\tau)=\tau$, and non-trivial examples can be constructed as follows: Let $f$ be an automorphic form of weight $k\in\BQ$ for $\G$, then
 \[
 h_f(\tau)=\tau +k\frac{f(\tau)}{f'(\tau)}
 \]
 is an equivariant function for $\G$. However, not every equivariant function for $\G$ arises in this way from an automorphic form \cite{rational}. Furthermore, it can be proven that the space of equivariant functions without the trivial element $h(\tau)=\tau$ is an infinite dimensional vector space with the zero element given by $h_{\Delta}$ \cite{structure}.

  Another interesting aspect involves the projectively  invariant cross-ratio
 defined for four distinct complex numbers $z_i$, $1\leq i\leq 4$ by
 \[
 [z_1,z_2,z_3,z_4]=\frac{(z_1-z_2)(z_4-z_3)}{(z_1-z_3)(z_4-z_2)}.
 \]
 In fact, the Schwarz derivative is the infinitesimal counterpart of the cross-ratio. Furthermore, by  definition, the cross-ratio of four distinct equivariant functions for $\G$ is an automorphic function for $\G$. As an illustration, one can check that
\[
[\tau,h_{\theta_2},h_{\theta_3},h_{\theta_4}]=\lambda
\]
and
\[
[\tau, h_{E_4},h_{\Delta},h_{E_6}]=J.
\]
Other arithmetical applications of equivariant functions can be found in \cite{critical}.

For the rest of this section, we suppose that $\G$ is the modular group $\SL$.  Let $I_2$ be the identity in $\SL$ and set
\[
T=\binom{1\ \ \ 1}{0\ \ \ 1}\,,\ \ S=\binom{0\  -1}{1\ \ \ 0}\,,\ \ P=ST=\binom{0\  -1}{1\ \ \ 1}.
\]
If $\gamma\in\GC$, we denote by $\ol{\gamma}$ its projection in $\PG$. We have $\ol{S}^2=\ol{P}^3=\ol{I}_2$,  $\ol{S}$ and $\ol{T}$ generate $\PS$, and so do $\ol{S}$ and $\ol{P}$. We now proceed to a reduction that will be useful in the following sections.

\begin{prop}
Let $h$ be a $\rho-$equivariant for $\PS$ such that $\rho(\ol{T})=\ol{T}$. Then there exists $L\in\GC$ such that $L\cdot h$ is equivariant for $\PS$. 	
\end{prop}
\begin{proof}
	Suppose that $h$ is a $\rho-$equivariant function such that $\rho(\ol{T})=\ol{T}$, that is,
	$h(\tau+1)=h(\tau)+1$ for $\tau\in\fH$.
	 Set $\rho(\ol{S})=\ol{M}$ for some $M\in \SLC$. Then we have $\rho(\ol{S^2})=\ol{I_2}=\ol{M^2}$. As $\det M=1$, we have  $M^2=\pm I_2$.
	
	If $M^2= I_2$, then $M=M^{-1}$, and as $M\in \SLC$, we should have  $M=\pm I_2$. Hence, $\rho(\ol{S})=\ol{M}=\ol{I_2}$. However,
	\[
	\ol{I_2}= \rho(\ol{ST})^3=\rho(\ol{T})^3=\ol{T^3}
	\]
	which is impossible. Therefore, we have
	\[\rho(\ol{S})=\ol{M} \mbox{ with } M^2= -I_2.
	\]
	
	Set $\displaystyle M=\mat$. Since $M^{-1}=-M\in \SLC$, we have $a=-d$ and $a^2+bc=-1$. If $\displaystyle N=MT=\binom{a\,\ a+b}{c\,\ c-a}$, then $\rho(\ol{P})=\ol{N}$ and so $N^3=\pm I_2$. We claim that $c\neq 0$, otherwise, we would have $a^2=-1$ and   $N^3=\binom{a^3\,\ *}{0\,\ -a^3} $, which implies that  $a^3=\pm 1$; a contradiction.

	If $\rho_1= L^{-1}\rho L$, with $\displaystyle L=\binom{1\,\ a/c}{0\ \ \ 1}$, then $L\cdot h$ is $\rho_1-$equivariant. Furthermore,
	\[\rho_1(\ol{T})=\ol{T}\  \mbox{ and }\  \rho_1(\ol{S})=\ol{M'} \ \mbox{ with }\ M'={\binom{0\ \ -1/c}{c \qquad \ 0}}.
	\]
Now let $\displaystyle N'=M'T=\binom{0\ \, -1/c}{c\ \qquad   c}$ so that $\rho_1(\ol{P})=\ol{N'}$. Then
\[ N'^3=\binom{-c\ \ \ \quad 1/c}{-c+c^2\ \ -c}= \pm I_2,
\]
and therefore $c=\pm1$, so that $\rho_1(\ol{S})=\ol{S}$. We conclude that $\rho_1$ is the defining representation of $\PS$, that is, $\rho_1(\gamma)=\gamma$ for all $\gamma\in\PS$ and that $L\cdot h$ is  equivariant for $\PS$.
	\end{proof}

\begin{cor}\label{cor2.2}
	Let $F$ be a weight 4 meromorphic modular form for $\SL$. If the equation
$\{h,\tau\}=F(\tau)$ has a solution such that $h(\tau+1)=h(\tau)+1$, then it has also a solution that is equivariant for $\SL$.	
\end{cor}
\begin{proof}
If $h$ is a solution such that $h(\tau+1)=h(\tau)+1$, then it is
$\rho-$equivariant with $\rho(\ol{T})=\ol{T}$. From the above proposition, there exists $L\in\GC$ such that $L\cdot h$ is equivariant. The Corollary follows since $h$ and $L\cdot h$ have the same Schwarz derivative.
	\end{proof}

\section{Quasi-modular solutions }

We now focus on the differential equation
\begin{equation}\label{r-equ}
y''+ \pi^2 r^2\,E_4\,y=\,0,
\end{equation}
where $r$ is a positive integer. The associated Schwarzian equation is thus
\begin{equation}\label{s-equ}
\{h,\tau\}\,=\,2\pi^2r^2\,E_4(\tau).
\end{equation}
The purpose of this section is to show that \eqref{r-equ} has always a solution that is quasi-modular and that \eqref{s-equ} has always a solution that is equivariant.

We set $q=\exp(2\pi i\tau)$ and the equation \eqref{r-equ} takes the shape
\[
\frac{\mbox{d}^2 y}{\mbox{d}q^2}+\frac{1}{q}\frac{\mbox{d}y}{\mbox{d}q}-\frac{r^2}{4}\frac{E_4(q)}{q^2}y=0,
\]
with $q$ in the punctured disc $\{0<|q|<1\}$. As $E_4(q)=1+\rm{O}(q)$, this is a Fuchsian differential equation with a regular singular point at $q=0$. To write down two linearly independent solutions, we apply the Frobenius method: The indicial equation is given by $x^2-r^2/4=0$ with solutions $x_1=r/2>0$ and $x=-r/2<0$. Thus, $x_1-x_2=r\in\BZ$, hence two linearly independent solutions are given by
\begin{equation}\label{sol-y1}
	y_1(\tau)=q^{r/2}\,\sum_{n=0}^{\infty}\,a_n\,q^n\,,\ \mbox{ with } a_0\neq 0,
\end{equation}
and
\begin{equation}\label{sol-y2}
	y_2(\tau)= \tau y_1(\tau) + q^{-r/2}\,\sum_{n=0}^{\infty}\,b_n\,q^n\,,\ \mbox{ with } b_0\neq 0.
\end{equation}
Moreover, a solution to \eqref{s-equ} is given by
\begin{equation}\label{sol-log}
			h(\tau)=\frac{y_2(\tau)}{y_1(\tau)}=\tau\, +\, q^{-r}\,\sum_{n=0}^{\infty}\,c_nq^n\,,\ \mbox{ with } c_0\neq 0.
\end{equation}
\begin{prop}\label{hr}
	The equation \eqref{s-equ} has a solution $h_r$ that is equivariant for $\SL$.
\end{prop}
\begin{proof}
	Since $r$ is an integer, the solution \eqref{sol-log} clearly satisfies $h(\tau+1)=h(\tau)+1$. Therefore, there exists a solution $h_r$ that is equivariant according to \corref{cor2.2}.
	\end{proof}
Recall that the modular group has a unique normal subgroup of index 2 denoted by $\SL^2$, which consists of the squares of the elements of $\SL$. In fact, $\SL^2=\langle P,Q\rangle$ with $P=\disp \binom{0\  -1}{1\ \quad \ 1}$ and $Q=S^{-1}PS=\disp \binom{1\  -1}{1\ \quad \ 0}$ both of order 6 in $\SL$, see \cite[page 16]{rankin}. In addition, $\SL^2$ is  a congruence group of level 2.
Its projection $\ol{\SL}^2$
is an index 2 normal subgroup of $\PS$ that is genus 0, has one cusp and two classes of elliptic  elements  $\ol{P}$ and $\ol{Q}$  both of order 3 in $\PS$.

 The following structure theorem provides solutions to \eqref{r-equ} that have specific modular properties.
\begin{thm} Let $r$ be a positive integer and let $\G=\SL$ if $r$ is even and $\G=\SL^2$ if $r$ is odd.
	 Then there exist two linearly independent solutions $f_1$ and $f_2$ to \eqref{r-equ} such that:
	\begin{enumerate}
		\item $f_1$ and $f_2$ have the following $q-$expansions
		\[
		f_1(\tau)=q^{r/2}\,\sum_{n\geq 0}\, \alpha_n q^n\,, \ \alpha_0\neq 0
		\]
		and
		\[
		f_2(\tau)=\tau f_1(\tau)+q^{-r/2}\,\sum_{n\geq 0}\, \beta_n q^n\,,\  \beta_0\neq 0.
		\]
		\item $f_1$ is a quasi-modular form  of weight zero and depth one for $\Gamma$ and $f_2(\tau)-\tau f_1(\tau)$ is a modular form of weight $-2$ for $\Gamma$.
		\item $h_r=f_2/f_1$ satisfies $h_r(\tau+1)=h_r(\tau)+1$. 
	\end{enumerate}	
\end{thm}
\begin{proof}
	Let $y_1$, $y_2$ as in \eqref{sol-y1} and \eqref{sol-y2}, $h=y_2/y_1$, and $h_r$ as in \propref{hr}. Also, let $L\in\SLC$ such that $h=L\cdot h_r$.
	If $\displaystyle Y=\binom{y_2}{y_1}$ then, for $\disp \gamma=\binom{a\ \ b}{c\ \ d}\in\SL$, the components of $ (c\tau+d)^{-1}Y(\gamma\tau)$ are also
	solutions to \eqref{r-equ}. Therefore,
\begin{equation}\label{vaf-1}
Y(\gamma \tau)\,=\,(c\tau+d)^{-1}\,\sigma(\gamma)\,Y(\tau)\,,\ \tau\in\fH\,,\ \gamma=\binom{a\ \ b}{c\ \ d}\in\SL\,,
\end{equation}
where $\sigma$ is the monodromy representation of $\SL$.
It is clear that
\[
\ol{\sigma(\gamma)}=\ol{L\gamma L^{-1}}, \, \gamma\in\SL,
\]
which is equivalent to having
\[
\sigma(\gamma)= \chi(\gamma) L\gamma L^{-1}, \ \gamma\in\SL,
\]
where $\chi $ is a character of $\SL$.
Now, if we let $\displaystyle F=L^{-1}Y= \binom{f_2}{f_1}$, then we have
 \begin{equation}\label{vaf-2}
F(\gamma \tau)\,=\,(c\tau+d)^{-1}\,\chi(\gamma) \gamma\,F(\tau),\ \tau\in\fH,\ \gamma=\binom{a\ \ b}{c\ \ d}\in\SL.
\end{equation}
Hence, $F$ is a vector-valued modular form of weight -1. In particular, this implies that $\chi(-\gamma)=\chi(\gamma)$ and so $\chi$ descends to a character on $\PS$. Thus the order of $\chi$ divides 6 as the commutator subgroup of $\PS$  has index 6.
Now, set $\disp L^{-1}=\binom{a_1\ \ b_1}{c_1\ \ d_1}$ , then on one hand we have
\[
f_1=c_1y_2+d_1y_1= (c_1 \tau + d_1)y_1+ c_1\, q^{-r/2}\, \sum_{n=0}^{\infty}\,b_n\,q^n,
\]
and on the other hand, we have $\displaystyle F(\tau +1)= \chi(T) T\, F(\tau)$ which implies that $f_1(\tau +1)= \chi(T) \, f_1$, and so $f_1$ is a periodic function as $\chi$ has a finite order. It follows that $c_1$ must be equal to zero, and that
\[f_1(\tau)=d_1 y_1(\tau)=q^{r/2}\,\sum_{n\geq 0}\, \alpha_n q^n\, \ \alpha_0\neq 0.
\]
 Therefore, $\displaystyle \chi(T)=(-1)^r$.
If we set $\G=$Ker$(\chi)$, then $\G=\SL$ if $r$ is even, and $\G=\SL^2$ if $r$ is odd.
 Now let  $\gamma=\binom{a\ \ b}{c\ \ d}\in\G$ and $\tau\in\fH$, then from \eqref{vaf-2} we get
\[
f_1(\g\tau)=\frac{d}{c\tau+d}\, f_1(\tau) + \frac{c}{c\tau+d}\, f_2(\tau)\,=\,
 f_1(\tau) + \frac{d}{c\tau+d}(f_2(\tau)-\tau f_1(\tau)).
\]
 Therefore $f_1$ is a quasi-modular form of  weight zero and depth one for $\G$. As a consequence,  $f_3(\tau)= f_2(\tau)-\tau f_1(\tau)$ is a modular from of weight $-2$ for $\G$.
Furthermore, we have
 \[
 f_3(\tau)=f_2(\tau)-\tau f_1(\tau)= [(a_1-d_1)\tau +b_1]y_1 + a_1\,q^{-r/2} \sum_{n=0}^{\infty}\,b_n\,q^n.
\]
Since $f_3$ is a periodic function we must have $a_1=d_1$. As $\det L^{-1}=a_1d_1=1$, we may assume that $a_1=1$. Thus
\[
 f_3(\tau)= q^{-r/2}\sum_{n=0}^{\infty}\,b_n\,q^n\, + b_1 y_1= q^{-r/2} \sum_{n=0}^{\infty}\,\beta_n\,q^n\,,\ \ \beta_0\neq 0.
\]
Finally, it is clear that $h_r=f_2/f_1$ satisfies $h_r(\tau+1)=h(\tau)+1$.
\end{proof}

\section{Explicit  solutions}

In this section we will provide a simple algorithm to explicitly construct solutions to \eqref{r-equ}  and equivariant solutions to \eqref{s-equ} for each $n\in \BN$. Through this algorithm, each solution is obtained by solving an explicit linear system.

 Again, we let $\G=\SL$ if $r$ is even and $\G=\SL^2$ if $r$ is odd. For $k\in\BZ$, let $M'_k(\G)$  be the space of weight $k$ modular forms that are  holomorphic in $\fH$ and meromorphic at the cusp at $\infty$, and let $M_k(\G)$ (resp. $S_k(\G)$) be the subspace of those forms that are holomorphic at $\infty$ (resp. vanishing at $\infty$). Similarly, we denote by
 ${M'_k}^s(\G)$ and $M_k^s(\G)$ the corresponding   spaces of quasi-modular forms for $\G$ of weight $k$ and depth $s$.

We start with some elementary facts about $M'_{-2}(\G)$ and $M'_{2}(\G)$ that will be useful to our construction. If  $f\in M'_{-2}(\G)$ with a pole of order $n$, then $\Delta^n f\in M_{12n-2}(\G)$ in case $\G=\SL$, and $\Delta^{n/2}f\in M_{6n-2}(\G)$ when $\G=\SL^2$. Moreover, any element in $\M'_{-2}(\G)$ has a Fourier expansion in $p=q=\exp(2\pi i\tau)$ if $\G=\SL$ and in $p=q^{1/2}=\exp(\pi i\tau)$ if $\G=\SL^2$.
\begin{lem}\label{singular}
	With the above notations, we have
	\[
	\{\,\mbox{Singular part of }f\,|\, f\in\M'_{-2}(\G)\}=\BC[p^{-1}].
	\]
\end{lem}
\begin{proof}
	We need to show that for each finite set $\{a_n\in\BC\,|\,n_0\leq n\leq -1\}$, there exists $f\in \M'_{-2}(\G)$ such that
	\[
	f(\tau)\,=\,\sum_{n=n_0}^{-1}\,a_np^n\, +\, \mbox{O}(p).
	\]
	As $\G$ has genus zero, it has a Hauptmodul $t$ that can be normalized to have a simple pole at $\infty$ and the expansion
	\[
	t(\tau)\,=\,\frac{1}{p}\,+\,\mbox{O}(p).
	\]
	Start with $t_0(\tau)=E_4E_6/\Delta\in \M'_{-2}(\G)$ if $\G=\SL$ or $t_0(\tau)=E_4/\Delta^{1/2}\in \M'_{-2}(\G)$ if $\G=\SL^2$ which both have a simple pole at $\infty$. One can choose a polynomial $P\in\BC[X]$ (of degree $-n_0-1$) such that
	\[
	P(t)t_0(\tau)\,=\,\sum_{n=n_0}^{-1}\,a_np^n\, +\, \mbox{O}(p),
	\]
	and we take $f=P(t)t_0\in M_{-2}(\G)$.
	\end{proof}
\begin{lem}\label{int}
	Every $f\in M'_2(\G)$ is the derivative of a modular function (of weight 0) for $\G$.
\end{lem}
\begin{proof}
	Let $f\in M'_2(\G)$. Since $f$ is holomorphic in $\fH$ which is simply connected, the following integral is well defined
	\[
	F(\tau)=\int_i^{\tau}f(z)dz,
	\]
	where the integration is over any path joining $i$ and $\tau$. Then
	for $\gamma\in\G$,
	\[
	F(\gamma\tau)=\int_{i}^{\gamma\tau}f(z)dz=F(\tau)+\int_i^{\gamma i}f(z)dz
	\]
	since $f(z)dz$ is a $\Gamma-$invariant differential. Furthermore, the map
	$\disp\gamma\mapsto \int_i^{\gamma i}f(z)dz$ is  an additive character of $\G$ which must be trivial  since both $\SL$ and $\SL^2$ are generated by elliptic elements. It follows that $F(z)\in M'_0(\G)$.
	\end{proof}
\begin{remark}{\rm
	It follows from this lemma that every element in $ M'_2(\G)$ has a zero constant term in its Fourier expansion.}
\end{remark}

Fix an element $g\in M'_{-2}(\G)$ and  set
\[f_0(\tau)= \int_{i}^{\tau} \,g(z)  E_4(z) dz \,,\ \tau\in\fH.
\]
 According to \lemref{int}, $f_0$ is a modular function for $\G$.
 \begin{prop}
Let $f_1= \pi^2r^2f_0+ g'$. Then
$f_1''+\pi^2r^2E_4f_1$ belongs to $M'_4(\G)$.
 \end{prop}
 \begin{proof}
 	One can check the modularity of $f_1$ by direct calculations. Also, one can
 	easily check that $f_1$ is a quasi-modular form of weight0 and depth 1 with quasi-modular polynomial $\disp P_{f_1}(X)=f_1-2gX$. In addition, $f_1''$ is a quasi-modular form of weight 4 and depth 3 but with a simple quasi-modular polynomial given by $\disp P_{f_1''}(X)=f_1''+2\pi^2r^2E_4gX$, and this readily shows that
 	$f_1''+\pi^2r^2E_4f_1\in M'_4(\G)$.
 	\end{proof}

In what follows, we will show that for each $r$, one can choose $g\in M'_2(\G)$ such that $f_1''+\pi^2r^2E_4f_1$ vanishes at infinity.
As $S_4(\G)$ is trivial for $\G=\SL$ or $\G=\SL^2$, we then  have a solution $f_1$ to \eqref{r-equ}. In the meantime, define
\begin{equation}\label{f2}
f_2(\tau)=-2g(\tau)+\tau f_1(\tau).
\end{equation}
Then
\[
f_2''-\tau f_1''=2(-g''+f_1')=2\pi^2r^2 gE_4=-\pi^2r^2 E_4(f_2-\tau f_1),
\]
so that
\begin{equation}\label{f1-f2}
f_2''+\pi^2r^2E_4f_2=\tau(f_1''+\pi^2r^2E_4f_1).
\end{equation}
 Thus, if $f_1$ is a solution to \eqref{r-equ}, then so is $f_2$. In addition, $f_1$ and $f_2$ are linearly independent, and we can take $\displaystyle h_r= f_2/f_1$ as a solution to \eqref{s-equ}.

\begin{thm}
	For each $r\in\BN$, one can construct $g\in M_{-2}'(\G)$ such that $f_1$ and $f_2$ defined above are two linearly independent solutions to \eqref{r-equ}.
\end{thm}
\begin{proof}
If $r$ is even, then $\G=\SL$ and we set $p=\exp(2i\pi \tau)$ and $\alpha=2 i\pi$, while if $r$ is odd, then $\G=\SL^2$ and we set $p=\exp(i\pi \tau)$ and $\alpha=i\pi$. For $g\in M_{-2}'(\G)$, write
\[
g(\tau)=\sum_{n\geq n_0}\, a_np^n\,,  \mbox{ and }\ E_4(\tau)=\sum_{n\geq 0}\,b_np^n,
\]
so that
\[
g(\tau)E_4(\tau)=\sum_{n\geq n_0}\,\sum_{s+t=n}\,a_s b_t\,p^n.
\]
Since $f_1=\pi^2r^2f_0+g'$, we have
	\begin{equation}\label{coeff}
		 f_1(\tau)= c+ \frac{\pi^2 r^2}{\alpha} \sum_{n\geq n_0, n\neq 0}\frac{1}{n}\sum_{s+t=n}a_s b_t\,p^n + \alpha \sum_{n\geq n_0}\,n a_n\,p^n,
	\end{equation}
	where $c$ is a constant. To eliminate the singular part of $f_1$, we must have
\[
\frac{\pi^2 r^2}{\alpha n}\sum_{s+t=n}a_s b_t \,+\, \alpha\, n\,  a_n = 0\,,\ \   n_0 \leq n \leq -1,
\]
that is,
\begin{equation}\label{sys}
  \frac{-\pi^2 r^2}{\alpha^2 n^2}\sum_{s=n_0}^{n}a_s b_{n-s} \,=\,   a_n  \hspace{1cm}  n_0 \leq n \leq -1.
\end{equation}

Write  $\displaystyle X=(a_{-1},\ldots a_{-n_0})^t$ and set, for $ 1 \leq k \leq -n_0 $ and $ 1 \leq l \leq -n_0$,
\[
B_{k\, l}= \frac{-\pi^2 r^2}{\alpha^2 k^2} b_{l-k}\, \mbox{ if } \,l\geq k\, \mbox{ and 0 otherwise}
\]
Then the linear system \eqref{sys} is equivalent to $\displaystyle BX = X$ where $B$ is the upper triangular matrix $B=(B_{k,l})$. Since $b_0=1$, the last entry of the diagonal is given by
\[
B_{-n_0,-n_0}=\frac{-\pi^2r^2}{\alpha^2n_0^2}.
\]
We now choose $n_0=-r/2$ if $r$ is even, and $n_0=-r$ if $r$ is odd so that
$B_{-n_0,-n_0}=1$. Therefore 1 is an eigenvalue of $B$ and a nontrivial eigenvector $X$ exists. Since the diagonal entries of the matrix $B$ are all distinct from one, except the last entry, we must have $a_{-n_0}\neq 0$. Otherwise the eigenvector $X$ will be zero, which is impossible. Hence $X$ is uniquely determined if we require that $a_{-n_0}=1$. According to \lemref{singular}, there exists a unique $g\in M_{-2}'(\G)$ such that
\[
g(\tau)\,=\,\sum_{n=n_0}^{-1}\,a_n\,p^n\, + \mbox{O}(1)\,,\ \  a_{n_0}=1.
\]
Now,  if $c=f_1(\infty) =0$, the proof is complete. However, if $c\neq 0$, let $\displaystyle F_1(\tau)=f_1(\tau) - c$ and $\displaystyle F_2(\tau)= -2g(\tau)+\tau F_1(\tau)$, $\tau \in \fH$.
We see that $F_1(\infty) =0$, and similarly to \eqref{f1-f2}, we have
\begin{equation}\label{F1-F2}
	F_2''+ \pi^2 r^2\,E_4\,F_2 \,=\, \tau (F_1''+ \pi^2 r^2\,E_4\,F_1).
\end{equation}
In the meantime,
\[
F_1''+ \pi^2 r^2\,E_4\,F_1 = f_1''+ \pi^2 r^2\,E_4\,f_1 - c \pi^2 r^2\,E_4 \in M_4(\G).
\]
Since $F_1(\infty) =0$, we have  $\displaystyle F_1''+ \pi^2 r^2\,E_4\,F_1 \in S_4(\G) =\{0\}$, and so $F_2''+ \pi^2 r^2\,E_4\,F_2 \,=\, 0$ using \eqref{F1-F2}. Hence  $F_1$ and $F_2$ are the required linearly independent solutions to \eqref{r-equ}.

 	\end{proof}

\section{Examples}
  We will illustrate how the algorithm of constructing the solutions works for few values of $r$.
  We will use the Ramanujan identities from \secref{notation} at will in the following calculations. We also use Maple software for the needed calculations.

 \noindent {\bf \underline{The case $r=1$:}}

 In this case, the weight -2 modular form $g$ with a simple at $\infty$ is unique up to a constant factor. This is due to the fact that $\dim M_4(\SL^2)=1$.
 We take $\displaystyle g_1=E_4/\Delta^{1/2}$. From the $p-$expansion, the modular function $\disp f_0(\tau)=\int_i^{\tau}\,\frac{E_4^2(z)}{\Delta^{1/2}(z)}\,dz$ should be a degree one polynomial of the Hauptmodul $E_6/\Delta^{1/2}$. As $E_6(i)=0$, we get
 \[
 f_0(\tau)=\int_i^{\tau}\,\frac{E_4^2(z)}{\Delta^{1/2}(z)}\,dz=\frac{-1}{i\pi}\frac{E_6}{\Delta^{1/2}}.
   \]
    Therefore, the solutions to $y''+\pi^2 E_4y=0$ are given by
    \[
    f_1=\pi^2 f_0+g_1'=\frac{i\pi}{3\Delta^{1/2}}(E_2E_4-E_6)=\frac{-1}{2}\frac{E_4'}{\Delta^{1/2}},
    \]
    and
    \[
    f_2=-2g_1+\tau f_1=-2\frac{E_4}{\Delta^{1/2}}-\frac{\tau}{2}\frac{E_4'}{\Delta^{1/2}}.
    \]
    Moreover, the equivariant solution to $\{h,\tau\}=2\pi^2E_4$  is given by
    \[
    h_1(\tau)=\tau+\frac{6/i\pi}{E_2-\frac{E_6}{E_4}}=\tau +\frac{4E_4}{E_4'}.
    \]

  \noindent {\bf \underline{The case $r=2$:}} \\

    In this case we also have a unique weight $-2$ form given by
    $\disp g_2=E_4E_6/\Delta$ because $M_{10}(\SL)$ is one-dimensional. We also have
    \[
    \int_i^{\tau}\,\frac{E_4^2E_6}{\Delta}\,dz =\frac{-1}{2\pi i}\left(\frac{E_4^3}{\Delta}-1728\right).
    \]
    Therefore, a solution to $y''+(2\pi)^2 E_4y=0$ with no constant term is given by
   \[
    f_1=4\pi^2f_0+g_2'=\frac{\pi i}{6\Delta}(6E_4^3-2E_2E_4E_6-4E_6^2)-1488\pi i.
   \]
   The second solution is given by $f_2=-2g_2+\tau f_1$. The corresponding Schwarzian equation $\{h,\tau\}=8\pi^2E_4$ has the equivariant solution
   \[
   	h_2(\tau)=\tau +\frac{6/i\pi}{E_2-\frac{E_6}{E_4}-\frac{720\Delta}{E_4E_6}}.
   \]

   \noindent {\bf \underline{The case $r=3$:}} \\

   The situation is a little bit different as $M_{16}(\SL^2)$ is two-dimensional. The $B-$matrix is given by
   \[
   B=\begin{bmatrix}
   	9&0&2160\\0&9/4&0\\0&0&1
   \end{bmatrix},
   \]
   with the eigenvector for the eigenvalue 1 given by
   \[X=\begin{bmatrix}
   	-270\\0\\1
   \end{bmatrix}.
   \]
   We look for the weight $-2$ form of the form
   \[
 g_3=\frac{E_4^4}{\Delta^{3/2}}+xg_1,
   \]
   where $g_1$ come from the case $r=1$. This allows us to use the integration   in the previous cases. The value of the parameter $x$ is such that the principal part of $g_3$ corresponds to the eigenvector $X$.
   We find
   \[
   g_3=\frac{E_4^4}{\Delta^{3/2}}-1226g_1=\frac{E_4^4}{\Delta^{3/2}}-1226\frac{E_4}{\Delta^{1/2}}.
   \]
   From the $p-$expansion, a primitive of $\disp E_4^5/\Delta^{3/2}$ should be a polynomial of degree 3 of the Hauptmodul $E_6/\Delta^{1/2}$. One is led to solve a simple linear system giving
   \[
   \int_i^{\tau}\frac{E_4^5}{\Delta^{3/2}}\,dz=\frac{1}{2\pi i}\left(-\frac{2}{3}\frac{E_6^3}{\Delta^{3/2}}-3456\frac{E_6}{\Delta^{1/2}}\right).
   \]
   The first solution to $y''+(3\pi)^2E_4y=0$ is given by
   \[
   f_1=\frac{\pi i}{3}\, \frac{9E_6^3+15006E_6\Delta-E_2E_4^4-8E_4^3E^6+1266E_2E_4\Delta}{\Delta^{3/2}}.
   \]
    The second solution is given by $f_2=-2g_3+\tau f_1$. The  equivariant solution to the corresponding Schwarzian equation is given by
   \[
   h_3(\tau)=\tau+\frac{6/i\pi}{E_2-\frac{E_6}{E_4}-\frac{720\Delta}{E_4E_6}+
   	\frac{95800320\Delta^2}{77E_4^4E_6+211E_4E_6^3}}.
   \]

   \noindent {\bf \underline{The case $r=4$:}} \\

The $B-$matrix is given by
\[
\begin{bmatrix}
	4&960\\0&1
\end{bmatrix}
\]
and an eigenvector for the eigenvalue 1 is given by $\disp X=\begin{bmatrix}
	-320\\1
\end{bmatrix}$. To determine the weight -2 modular form whose principal part correspond to $X$. A simple calculation shows that
\[
g_4=\frac{E_4^4E_6}{\Delta^2}-824g_2=\frac{E_4^4E_6}{\Delta^2}-824\frac{E_4E_6}{\Delta}.
\]
The first solution to $y''+(4\pi)^2E_4y=0$ is given by
\[
f_1=\frac{-1}{648\Delta}(219E_4^6-641E_4^3E_6^2+113E_2E_4^4E_6+103E_2E_4E_6^3+206E_6^4),
\]
and the second solution is $f_2=-2g_4+\tau f_1$. As for the corresponding Schwarzian equation, the equivariant solution is given by
\[
h_4=\tau+\frac{6/i\pi}{E_2-\frac{E_6}{E_4}-\frac{720\Delta}{E_4E_6}+
	\frac{95800320\Delta^2}{77E_4^4E_6+211E_4E_6^3}-\frac{9146248151040\Delta^3}{E_4E_6(8701E_4^6+31774E_4^3E_6^2+21733E_6^4)}}.
\]
We notice that for each $r$, the form $g_r$ of weight -2 involves the previous $g_i$ with $i$ of the same parity as $r$. A simple explanation comes from the $p-$expansions of the previous sections, and also from the dimension formulas for the spaces $M_{12k-2}(\G)$. 
Also, we notice how the equivariant solutions $h_r$ are related. This suggest that one should be able to write down a Maple code to produce inductively the solutions for each $r$.

As for the solutions $h_r$ to the Schwarzian equation, they have the shape
\[
h_r(\tau)=\tau+\frac{6/i\pi}{E_2+f}\,,
\]
where $f$ is a weight 2 meromorphic modular form. In addition,  $\disp \frac{i\pi}{6}(E_2+f)$ is a quasi-modular form of weight 2 and depth 1 and whose quasi-modular polynomial is monic, and therefore $h_r$ is readily equivariant with no need of conjugation \cite{structure}.



\begin{thebibliography}{aaaa}
	\bibitem{structure} A. Elbasraoui; A. Sebbar. Equivariant forms: Structure and geometry. Canad. Math. Bull. Vol. {\bf 56} (3), (2013) 520--533.
	\bibitem{rational} A. Elbasraoui; A. Sebbar. Rational equivariant forms. Int. J. Number Th. 08  No. 4(2012), 963--981.
	\bibitem{ford} L. R. Ford. Automorphic functions. McGraw-Hill 1929
	\bibitem{mason-franc} C. Franc; G. Mason.  Hypergeometric series, modular linear differential equations and vector-valued modular forms. Ramanujan J. 41 (2016), no. 1-3, 233--267.
	 \bibitem{grabner} P. J. Grabner. Quasimodular forms as solutions of modular differential equations.  Int. J. Number Theory 16 (2020), no. 10, 2233--2274.
	 	\bibitem{hurwitz} A. Hurwitz, Adolf: Ueber die Differentialgleichungen dritter Ordnung, welchen
	\bibitem{ka-ko} M. Kaneko; M. Koike, On modular forms arising from a differential equation of hypergeometric
	type. Ramanujan J. 7(2003), no. 1-3, 145--164.
	\bibitem{ka-et-al} M. Kaneko; K. Nagatomo; Y. Sakai. The third order modular linear differential equations. J. Algebra 485 (2017), 332--352.
	\bibitem{ka-za} M. Kaneko; D. Zagier. Supersingular j-invariants, hypergeometric series, and Atkin's orthogonal polynomials. Computational perspectives on number theory (Chicago, IL, 1995), 97--126, AMS/IP Stud. Adv. Math., 7, Amer. Math. Soc., Providence, RI, 1998.
	\bibitem{klein} F. Klein,  Ueber Multiplicatorgleichungen. (German) Math. Ann. 15 (1879), no. 1, 86--88.
	\bibitem{milas} A. Milas, Ramanujan’s “Lost Notebook” and the Virasoro algebra. Comm. Math. Phys. 251(2004),
	no. 3, 657--678.
	\bibitem{mukhi} S. Mathur, S. Mukhi, and A. Sen, On the classification of rational conformal field theories. Phys. Lett.
	B 213(1988), no. 3, 303--308.
	\bibitem{mathann} J. McKay; A. Sebbar.  Fuchsian groups, automorphic functions
	and Schwarzians. Math. Ann. 318 (2), (2000) 255--275.
	\bibitem{mason} G. Mason, 2-dimensional vector-valued modular forms,
	Ramanujan J (2008) 17: 405--427.
	\bibitem{nakaya} T. Nakaya. On modular solutions of certain modular differential equation and supersingular polynomials. Ramanujan J. 48 (2019), no. 1, 13–-20. 
	\bibitem{ram} S. Ramanujan, On certain arithmetical functions. Trans. Cambridge Philos. Soc. 22(1916), 159–184.
	\bibitem{rankin} R. Rankin. Modular Forms and Functions, Cambridge Univ. Press, Cambridge, 1977.
	
	\bibitem{critical} A. Sebbar; H. Saber. On the critical points of modular forms.  J. Number Theory 132 (2012), no. 8, 1780–1787.
	\bibitem{vvmf} A. Sebbar; H. Saber. Equivariant functions and vector-valued modular forms. Int. J. Number Theory 10 (2014), no. 4, 949--954.
	\bibitem{kyushu} A. Sebbar; H. Saber. On the existence of vector-valued automorphic forms. Kyushu J. Math. 71 (2017), no. 2, 271--285.
	\bibitem{forum}  A. Sebbar; H. Saber. Automorphic Schwarzian equations.  Forum Math. DOI: 10.1515/forum-2020-0025
	\bibitem{ramanujan} Automorphic Schwarzian equations and integrals of weight 2 forms. The Ramanujan J. https://doi.org/10.1007/s11139-020-00348-w.
	\bibitem{shimura} G. Shimura; Introduction to the Arithmetic Theory of Automorphic Functions, Princeton University Press, Princeton, New Jersey, 1971.
\end{thebibliography}
\end{document}